\documentclass{amsart}
\usepackage{graphicx,euscript,amsmath,bm,enumerate}
\usepackage[alphabetic]{amsrefs}
\pdfoutput=1

\newtheorem{theorem}{Theorem}[section]
\newtheorem{lemma}[theorem]{Lemma}

\theoremstyle{definition}
\newtheorem{remark}{Remark}

\newtheorem*{lemma*}{Lemma}

\newcommand{\myfig}[3][]{
 \begin{figure}
 \begin{center}
 {\mbox{\includegraphics[#1]{#2.pdf}}}
 \end{center}
 \caption{\label{#2}#3}
 \end{figure}}

\def\vc#1{\mathbf #1}

\def\O{{\mathcal O}}
\def\F{{\mathcal F}}
\long\def\ignore#1{}
\def\tt{\theta}

\def\cc{\gamma}

\def\H{{\mathcal H}}

\def\nd{\noindent}
\def\fref#1{Figure~\ref{#1}}

\def\ss{\sigma}

 \providecommand{\Pic}{\mathop{\rm Pic}\nolimits}
 
\def\A{{\mathcal A}}
\def\Bbb#1{{\mathbb #1}}

\def\E{{\mathcal E}}

\def\CC{{\Gamma}}
\def\GG{{\Gamma}}

\def\LL{\Lambda}

\def\K{{\mathcal K}}

\def\L{{\mathcal L}}

\def\proj{{\rm{proj}}}

\def\H{{\mathcal H}}
\def\Ap{{\rm Ap}}
\def\myspan{{\rm span}}

\renewcommand{\odot}{\mathbin{\mathchoice
  {\xcirc\scriptstyle}
  {\xcirc\scriptstyle}
  {\xcirc\scriptscriptstyle}
  {\xcirc\scriptscriptstyle}
}}
\newcommand{\xcirc}[1]{\vcenter{\hbox{$#1\circ$}}}

\def\Viete{Vi\`ete }
\def\P{{\mathcal P}}

\begin{document}

\title{Apollonian packings in seven and eight dimensions}
\author{Arthur Baragar}
\begin{abstract}
In an earlier work, we proposed a generalization for the Apollonian packing in arbitrary dimensions and showed that the resulting object in four, five, and six dimensions have properties consistent with the Apollonian circle and sphere packings in two and three dimensions.  In this work, we investigate the generalization in seven and eight dimensions and show that they too have many of the properties of those in lower dimensions.  In particular, the hyperspheres are tangent or do not intersect; they fill the hyperspace; the object includes a maximal cluster of mutually tangent hyperspheres; and there exists a perspective where all hyperspheres in the object have integer curvatures.
\end{abstract}
\address{Department of Mathematical Sciences, University of Nevada, Las Vegas, NV 89154-4020}
\email{baragar@unlv.nevada.edu}
\thanks{\nd \LaTeX ed \today.}

\maketitle

\section{Introduction}

In an earlier work \cite{Bar18}, we proposed an alternative definition for the Apollonian circle and sphere packings.  That definition  (see Section \ref{ss2.4}) generates the Apollonian circle and sphere packings in two and three dimensions, extends to any dimension, and in dimensions 4, 5, and 6, defines a configuration of hyperspheres with the following properties:
\begin{enumerate}[\indent (a)]
\item   The hyperspheres do not intersect or intersect tangentially.
\item  The hyperspheres fill $\Bbb R^N$.
\item  The configuration includes a cluster of $N+2$ mutually tangent hyperspheres.
\item  Every hypersphere in the configuration is a member of a cluster of $N+2$ mutually tangent hyperspheres.
\item  If every hypersphere in a cluster of $N+2$ mutually tangent hyperspheres has integer curvature, then every hypersphere in the configuration has integer curvature.  
\end{enumerate}
In this paper, we show that the configurations in dimensions 7 and 8 also have Properties (a) through (e).  The generalizations in all dimensions have Properties (c) and (e) \cite{Bar18}, so our main result is to demonstrate Properties (a), (b) and (d).   Let us call a configuration of hyperspheres that satisfies Properties (a) and (b) a {\it packing} or {\it Apollonian-like packing}.    We have long known of the existence of hypersphere packings in many dimensions (see e.g. \cite{Boy74}), but it was once believed that for dimension $N\geq 4$, none of them include clusters of $N+2$ mutually tangent spheres (see the Mathematical Review for \cite{Boy74}, and \cite[p.~356, last paragraph]{LMW02}).         

The above statements are a bit imprecise, so let us clarify.  Associated to each circle (or sphere, hypersphere) is a side  -- the inside or outside.  If that side were always the inside, then we would call them discs or balls, and some authors favor that nomenclature.   The {\it curvature}\footnote{Some authors prefer the term {\it bend}, since in higher dimensions, the usual definition of curvature is the inverse of the square of the radius.} of a circle (or sphere, hypersphere) is the inverse of its radius, together with a sign that is positive if the associated side is the inside, and negative otherwise.  So, for example, the outside circle in \fref{fig1} (left) has negative curvature, while all the rest have positive curvature.  By fill $\Bbb R^N$ we mean that there is no space left where we can place a hypersphere so that it does not intersect any hypersphere (or its associated side) in the packing.  So, for example, the outside circle in \fref{fig1} covers most of the plane, including the point at infinity, while the rest cover most of the remainder.  Note too that if a circle lies inside another but does not intersect it (as in the context of Property (a) above), then the curvature of the outside circle must be negative.  That is, two nested circles with curvatures of the same sign are said to intersect, since their associated sides have overlap.  This convention (or something like it) is also necessary in the formulation of Descartes' Theorem (see Theorem \ref{tDescartes} below).  When the curvature is zero, the circle is a line (or plane, hyperplane).  In the strip version of the Apollonian circle packing (see \fref{fig1}, right), it is clear which sides we wish to associate to the two lines, but the curvature is not informative in this case.  In the next section, we will give a different interpretation that clarifies this ambiguity.       

To show our main result, we first streamline the argument in dimensions 4, 5 and 6.  In passing, we introduce groups that are generated by reflections (unlike the descriptions given in \cite{Bar18}), so have Coxeter graphs.  These are shown in Section~\ref{sCoxeter}.  

\section{Background}

\subsection{The Apollonian circle packing}  \label{s2.1} To generate the Apollonian circle packing, we begin with four mutually tangent circles $\vc e_1$, ..., $\vc e_4$.  In each of the resulting curvilinear triangles (one of which might include the point at infinity), we inscribe a circle that is tangent to all three sides, thereby creating new curvilinear triangles.  We repeat this process indefinitely (see \fref{fig1}).  

\myfig[width=\textwidth]{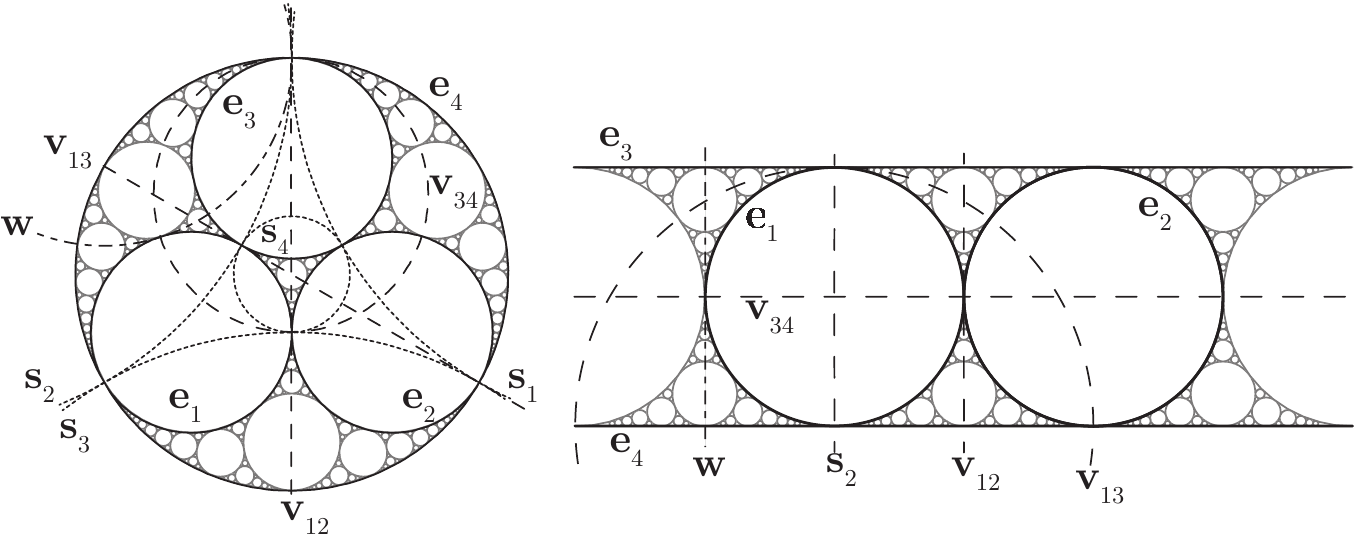}{Two versions of the Apollonian circle packing.  The one on the right is the {\it strip packing}.  The dotted and dashed lines represent symmetries of the packing.}

The process can be interpreted algebraically.  Let us select three of the initial four circles, say $\vc e_1, \vc e_2, \vc e_3$.  They form two curvilinear triangles, one of which contains $\vc e_4$.  They also have three points of tangency about which we can circumscribe a circle $\vc s_4$.   Let $R_{\vc n}$ denote inversion in the circle $\vc n$.  Note that the three circles $\vc e_i$ for $i\neq 4$ are fixed by $R_{\vc s_4}$, since $\vc s_4$ intersects $\vc e_i$ perpendicularly for $i\neq 4$.  Since $\vc e_4$ is tangent to the other three circles, its image must also be tangent to those three circles.  Thus, the image of $\vc e_4$ under the inversion in $\vc s_4$ is the incircle of the other curvilinear triangle.  We can similarly define $\vc s_1$, $\vc s_2$, and $\vc s_3$, giving us the {\it Apollonian group}
\[
\CC_{\Ap}=\langle R_{\vc s_1}, R_{\vc s_2}, R_{\vc s_3}, R_{\vc s_4}\rangle.
\]
The Apollonian packing is the image of the initial four circles $\vc e_1$,...,$\vc e_4$ under the action of $\CC_{\Ap}$.   Let us call the four inversions $R_{\vc s_i}$ {\it Vi\`ete involutions}.  As we will see, they fix three variables in a quadratic in four variables and send the fourth to its other root.  This can sort of be seen using Descartes' Theorem:

\begin{theorem}[Descartes] \label{tDescartes}  Suppose four mutually tangent circles have curvatures $k_1$,..., $k_4$ (with appropriate signs so that the circles do not intersect).  Then
\[
(k_1+k_2+k_3+k_4)^2=2(k_1^2+k_2^2+k_3^2+k_4^2).
\]
\end{theorem}  

If the curvatures are of the four initial circles, and a fifth is generated by one of the \Viete involutions, then its curvature is the other root of the quadratic generated by fixing three of the curvatures in Descartes' Theorem.  

The Apollonian sphere packing (or Soddy packing) can be generated by mimicking the above algebraic interpretation:  Begin with five mutually tangent spheres and define the five \Viete involutions as before.  That is, fix four of the five spheres and define a \Viete involution to be inversion in the sphere that is perpendicular to all four spheres (such a sphere exists).  The five \Viete involutions generate a group and the image of the initial five spheres under the action of this group is the Apollonian sphere packing.  

In four dimensions, using \Viete involutions leads to overlapping hyperspheres, as was noticed by Boyd \cite{Boy74}.   

The \Viete involutions are symmetries of the Apollonian circle packing, but they are not the only ones.  There are a class of symmetries we call {\it transpositions} that switch two of the original circles and fix the other two.  Geometrically, they are inversion in a circle $\vc v_{ij}$ that is tangent to both $\vc e_i$ and $\vc e_j$, and is perpendicular to the other two (see \fref{fig1}).  Their action on curvatures in Descartes' Theorem is to switch $k_i$ and $k_j$.  The group generated by the transpositions is the symmetric group $S_4$, so is finite.    

The full group of symmetries of the Apollonian circle packing is 
\[
\CC=\langle R_{\vc v_{12}}, R_{\vc v_{34}}, R_{\vc v_{14}}, R_{\vc s_2}\rangle,
\]
and the packing is the image of $\vc e_4$ under the action of $\CC$.   The full group must, of course, contain a \Viete involution, but it is possible to choose a subgroup that contains no \Viete involutions yet still generates the packing.  

Let us consider the subgroup 
\[
\CC'=\langle R_{\vc v_{12}}, R_{\vc v_{34}}, R_{\vc v_{14}}, R_{\vc w}\rangle, 
\]
where $\vc w$ is the circle/line shown in \fref{fig1}.  Then $\CC'$ is a subgroup of $\CC$ of index two, and the Apollonian packing is the image of $\vc e_4$ under the action of $\CC'$.  Thus, it is possible to generate the packing using a group that includes no \Viete involutions.  %For higher dimensions, we will describe a $\vc w$ and use it to generate a group that generates the configuration.  For $\rho=4$, though, the analog of $\vc w$ is actually $\vc s_2$.  
%It is this group that nicely generalizes to higher dimensions, but only to dimension 6.   Still, it streamlines the work in \cite{Bar18} and simplifies arguments for dimensions 7 and 8.    

\subsection{Hyperbolic geometry}  A circle in the Euclidean plane can be thought of as the edge of a hemisphere in the Poincar\'e upper half space model of $\Bbb H^3$, so represents a plane in $\Bbb H^3$.  The Apollonian packing therefore represents an infinite sided {\it ideal} polyhedron in $\Bbb H^3$.  (The polyhedron has no edges, like an ideal triangle in $\Bbb H^2$, which has no vertices.)  Similarly, an  $(N-1)$-sphere in $\Bbb R^N$ represents the edge of an $N$-dimensional hemisphere in the Poincar\'e upper half hyperspace model of $\Bbb H^{N+1}$.  The different versions of the Apollonian packing (e.g. those in \fref{fig1}) can therefore be thought of as the same object in $\Bbb H^3$, but rendered with different points at infinity.  We call them different {\it perspectives}.  One perspective can be obtained from another by inversion in some circle (or sphere, hypersphere).

Associated to an Apollonian packing is a lattice, which is hinted at with Property~(e).  This lattice lies in Lorentz space, in which we can imbed a hyperbolic geometry.

\subsection{The pseudosphere in Lorentz space}  Lorentz space, $\Bbb R^{\rho-1,1}$, is the set of $\rho$-tuples over $\Bbb R$ equipped with the Lorentz product
\[
\vc u\odot \vc v:=u_1v_1+u_2v_2+...+u_{\rho-1}v_{\rho-1}-u_{\rho}v_{\rho}.
\]
The surface $\vc x\odot \vc x=-1$ is a hyperboloid of two sheets.  Let us take the top sheet
\[
\H: \qquad \vc x\odot \vc x=-1, \qquad x_{\rho}>0.
\]
We define the distance $|AB|$ between two points on $\H$ by
\[
\cosh(|AB|)= -A\odot B.
\]
The pseudosphere $\H$ equipped with this metric is a model of $\Bbb H^{\rho-1}$.  (See \cite{Rat06}, or \cite{Bar18} for more details and references.)
   
Hyperplanes on $\H$ are the intersection of $\H$ with hyperplanes in $\Bbb R^{\rho-1,1}$ that go through the origin.  That is, hyperplanes of the form 
$\vc n\odot \vc x=0$ with $\vc n\in \Bbb R^{\rho-1,1}$.  The hyperplane intersects $\H$ if and only if $\vc n\odot \vc n>0$.  Let us denote the hyperplane in $\Bbb R^{\rho -1,1}$ and its intersection with $\H$ by $H_{\vc n}$.  The plane divides $\Bbb R^{\rho-1,1}$ and $\H$ into two halves, which we denote $H_{\vc n}^+$ and $H_{\vc n}^-$, where
\[
H_{\vc n}^+=\{\vc x: \vc n\odot \vc x\geq 0\}.
\]
The angle $\tt$ between two hyperplanes $H_{\vc n}$ and $H_{\vc m}$ that intersect in $\H$ is given by
\begin{equation}\label{eq1}
|\vc n||\vc m|\cos \tt = -\vc n\odot \vc m,
\end{equation}
where $|\vc n|=\sqrt{\vc n\odot \vc n}$, and $\tt$ is the angle in the region $H_{\vc m}^+\cap H_{\vc n}^+$.  %If the planes do not intersect, then
%\[
%|\vc n\odot \vc m|=|\vc n||\vc m|\cosh \psi
%\]
%where $\psi$ is the shortest distance between the two planes $H_{\vc m}$ and $H_{\vc n}$.  The sign of $\vc n\odot \vc m$ is negative if $H_{\vc m}^+\cap H_{\vc n}^+$ is the region between the two planes.  

%When $\vc u\odot \vc u<0$, our notation $|\vc u|=\sqrt{\vc u\odot\vc u}$ is the positive imaginary square root.  The notation $||\vc u||$ represents the absolute value of $|\vc u|$.

We let
\begin{align*}
\O(\Bbb R)&=\{T\in M_{\rho\times\rho}: T\vc u\odot T\vc v=\vc u\odot \vc v \hbox{ for all $\vc u,\vc v\in \Bbb R^{\rho-1,1}$}\} \\
\O^+(\Bbb R)&=\{T\in \O(\Bbb R): T\H=\H\}.
\end{align*}
Reflection through the plane $H_{\vc n}$ is given by
\begin{equation}\label{eq2}
R_{\vc n}(\vc x)=\vc x-2\proj_{\vc n}(\vc x) \frac{\vc n}{|\vc n|}=\vc x-2\frac{\vc n\odot \vc x}{\vc n\odot \vc n}\vc n,
\end{equation}
and is in $\O^+(\Bbb R)$.  (Our use of the same notation for inversion in the circle $\vc n$ is on purpose.)  
Because all isometries are generated by reflections, the group $\O^+(\Bbb R)$ is the group of isometries of $\H$.  
  
Let $\partial \H$ represent the boundary of $\H$, so $\partial \H$ is a $(\rho-2)$-sphere.  It is represented by $\L^+/\Bbb R^+$ where $\L^+$ is the set
\[
\L^+=\{\vc x\in \Bbb R^{\rho-1,1}: \vc x\odot \vc x=0, x_\rho>0\}.
\]
Given an $E\in \L^+$, the set $\partial \H_E=\partial \H\setminus {E\Bbb R^+}$ is a model of Euclidean space $\Bbb R^{\rho-2}$ using the metric $|AB|_E$ defined by
\[
|AB|_E^2=\frac{-2A\odot B}{(A\odot E)(B\odot E)}.
\]
This is the boundary of the Poincar\'e upper half hyperspace model of $\Bbb H^{\rho-1}$ using $E$ for the point at infinity.  The plane $H_{\vc n}$ represents a $(\rho-3)$-sphere in this model, which we denote with $H_{\vc n, E}$, or sometimes $H_{\vc n}$ if $E$ is understood or is not important, or sometimes just $\vc n$.  The reflection $R_{\vc n}(\vc x)$, when restricted to $\partial H_E$, is inversion in $H_{\vc n, E}$.  If $H_{\vc n}^+$ is the side we wish to associate to the hypersphere $H_{\vc n}$ then the curvature of $H_{\vc n, E}$ is given by
\[
\frac{-\vc n\odot E}{|\vc n|},
\]
using the metric $|\cdot |_E$ \cite{Bar18}.   

\begin{remark}  Boyd's representation of an $(N-1)$-sphere in $\Bbb R^N$ using an $(N+2)$-tuple (see \cite{Boy74}) is the same as our representation.  Boyd calls these vectors {\it polyspherical coordinates}, and attributes them to Clifford and Darboux.   
\end{remark}

\subsection{The formal definition} \label{ss2.4}  Let $\vc e_1$, ..., $\vc e_{\rho}$ be $\rho$ mutually tangent hyperspheres in $\Bbb R^{\rho-2}$.  Let us think of these as vectors in $\Bbb R^{\rho-1,1}$, normalized to have length 1 so $\vc e_i\odot \vc e_i=1$, and oriented so that $H_{\vc e_i}^-$ contains $H_{\vc e_j}$ for $j\neq i$.  Then by the tangency conditions and Equation (\ref{eq1}), $\vc e_i\odot \vc e_j=-1$ for $i\neq j$.  The matrix
\[
J_{\rho}=[\vc e_i\odot \vc e_j]
\]
has $1$'s along the diagonal, and $-1$'s off the diagonal.  It clearly has an eigenvalue of $2$ with multiplicity $\rho-1$, and an eigenvalue of $2-\rho$ with multiplicity one.  Thus $J_\rho$ is non-degenerate, which both proves that it is possible to have $\rho$ mutually tangent hyperspheres in $\Bbb R^{\rho-2}$, and that this is maximal.   We define the lattice
\[
\LL=\LL_{\rho}   =\vc e_1\Bbb Z\oplus \cdots \oplus \vc e_{\rho}\Bbb Z.
\]
We choose $D\in \LL$ so that $D\odot D<0$, $D/||D||\in \H$, and $D\odot \vc n\neq 0$ for any $\vc n\in \LL$ that satisfies $\vc n\odot \vc n=1$.  The choice $D=\vc e_1+...+\vc e_\rho$ works \cite{Bar18}, and we fix this choice through the rest of the paper.  Let 
\[
\E_1=\{\vc n\in \LL: \vc n\odot \vc n=1, \vc n\odot D<0\}
\]
and 
\[
\K_{\rho}=\bigcap _{\vc n\in \E_1}H_{\vc n}^-.
\]
That is, for each $\vc n\in \LL$ with $\vc n\odot \vc n=1$, we consider the half space bounded by $H_{\vc n}$ that contains $D$ and take the intersection of all these half spaces.  Viewed in $\Bbb R^{\rho-1,1}$, this gives us a polyhedral cone with an infinite number of faces.  Viewed as an object in $\H$, it is a polyhedron with an infinite number of sides.  Its intersection with $\partial \H$ is the {\it Apollonian configuration} $\A_{\rho}$.  This is a departure in terminology from \cite{Bar18}, motivated by our definition of a packing given in the introduction.  Our main result is therefore to show that the Apollonian configurations for $\rho=9$ and $10$ are packings (and to show Property (d)).

Not every $\vc n\in \E_1$ gives a face of $\K$, so let us define
\[
\E_1^*=\{\vc n\in \E_1: \hbox{ $H_{\vc n}$ is a face of $\K_{\rho}$}\}. 
\]
Then 
\[
\A_{\rho}=\{H_{\vc n}\subset \H: \vc n\in \E_1^*\}    
\]
Given $E\in \L^+$, the intersection with $\partial \H_E$ gives a {\it perspective},
\[
\A_{\rho,E}=\{H_{\vc n,E}: \vc n\in \E_1^*\}.
\]
The examples in \fref{fig1} are the perspectives $\A_{4,(1,1,1,3+2\sqrt 3)}$ and $\A_{4,(0,0,1,1)}$.

For fixed $\rho$, the object $\A_{\rho}$ exists, but {\it a priori} may not satisfy the properties outlined in the introduction.  In particular, the polyhedron may have edges, which would mean the hyperspheres in $\A_{\rho}$ intersect.  

\begin{remark}  By a result in \cite{Mor84}, for fixed $\rho\leq 10$, there exists a K3 surface $X$ so that the lattice $\LL_{\rho}$ is $\Pic(X)$ and the intersection matrix is $-2J_{\rho}$.  In this context, $D$ should be thought of as an ample divisor, so $\K_{\rho}$ is the ample cone for $X$.  The set $\E_1$ is the set of effective $-2$ divisors in $\Pic(X)$, while $\E_1^*$ is the set of irreducible effective $-2$ divisors.
\end{remark}

\begin{remark}  The set 
\[
\{H_{\vc n}: \vc n\in \E_1\}
\]
is the {\it Apollonian super-packing} described in \cite{GLM06}.
\end{remark}

\begin{remark}  The generalized Descartes Theorem is as follows:  Suppose $\rho$ hyperspheres are represented by linearly independent vectors $\vc e_1$, ..., $\vc e_\rho$ and have curvatures $k_1$, ..., $k_\rho$.  Let $J=[\vc e_i\odot \vc e_j]$ and $\vc k=(k_1,...,k_\rho)$.  Then $\vc k^tJ^{-1}\vc k=0$.  
This was observed by Boyd \cite{Boy74}.  %Note that, for the Apollonian circle packing, $J^{-1}$ is a multiple of $J$.  
\end{remark}  

\section{The Apollonian packings in dimensions $N\leq 6$}

Our strategy is to describe (for each $\rho$) a subgroup $G$ of finite index in $\O_{\LL}^+$ and identify a fundamental domain $\F=\F_G$ for $G$.  The fundamental domain will have finite volume and a finite number of sides, one of which will be a face of $\K$.  We identify that face and remove the appropriate generator from our description of $G$ to give us a thin group $\CC$.  Then $\K=\CC \F$.  Because our interests are in relatively high dimensions $\rho=9$ and $10$ (so $N=7$ and $8$), we describe the process in small dimensions first ($\rho=4$ and $5$) and build up from there.  

Throughout, we will consider strip packings with $E=\vc e_{\rho-1}+\vc e_\rho$ as our point at infinity.  Note that $E$ is the point of tangency between $\vc e_{\rho-1}$ and $\vc e_\rho$.  %(Verify that $E\odot \vc e_{\rho-1}=E\dot \vc e_\rho=0$.)  
As in \cite{Bar18}, we first find a fundamental domain $\F_{\{\vc e_{\rho-1},\vc e_{\rho}\} }$ for the stabilizer group $G_{\{\vc e_1,\vc e_2\} }$, extend it to a fundamental domain $\F_E$ for $G_E$, and then to a fundamental domain $\F$ for $G$.  Note that $G_E$ is a group of Euclidean isometries on $\partial \H_E\cong \Bbb R^{\rho-2}$, and $G_{\{\vc e_{\rho-1},\vc e_{\rho}\}}$ is a group of Euclidean isometries in $(\rho-3)$-dimensions on the intersection of $\partial H_E$ with $H_{\vc v_{\rho-1,\rho}}$.  

\subsection{The Apollonian circle packing again ($\rho=4$)}
We add $R_{\vc e_4}$ to $\CC$ to get our group $G$.  Its fundamental domain $\F$ is bounded by the planes $H_{\vc n}\subset \H$ for $\vc n\in \{\vc e_4,\vc s_2, \vc v_{34}, \vc v_{12}, \vc v_{13}\}$ (see \fref{fig2}).  The square highlighted in \fref{fig2} (left) is $\F_E$, and we call the region above it in the Poincar\'e upper half-space model the {\it chimney}.  The circle given by $\vc v_{13}$ represents a hemisphere (the plane $H_{\vc v_{13}}$) that we call the {\it dome}.   The fundamental domain $\F$ is the portion of the chimney above the dome.  It is a polyhedron in $\H$ that intersects $\partial \H$ at the single point $E$, so has finite volume.  Thus, $G$ has finite index in $\O_{\LL}^+$, provided it is in fact a subgroup of $\O_{\LL}^+$.   

\myfig{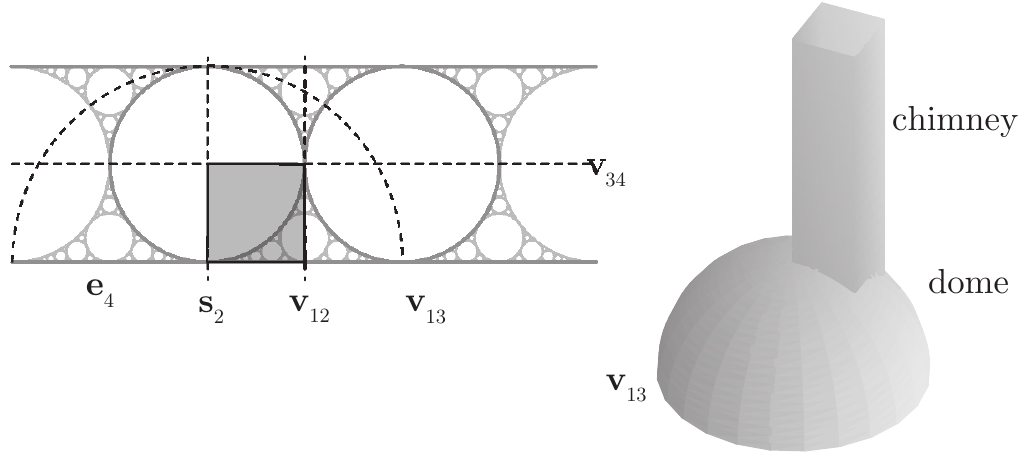}{The fundamental domain for $\O_{\LL}^+$ when $\rho=4$: The picture on the left is on the boundary of the Poincar\'e upper half-space model, so represents an object above it, which is shown on the right.  The fundamental domain is the portion of the chimney that is above the dome.}

Note that $\vc e_4\odot \vc e_4=1$, so $R_{\vc e_4}$ has integer entries (see Equation (\ref{eq2})) and hence is in $\O_{\LL}^+$.   The reflections $R_{\vc v_{ij}}$ switch the $i$-th and $j$-th components, which is why we call them transpositions.  For $R_{\vc s_2}$, we note that $\vc s_2\odot \vc e_i=0$ for $i\neq 2$ since $\vc s_2$ is perpendicular to $\vc e_i$ (for $i\neq 2$).  Solving, we find $\vc s_2=(1,-1,1,1)$ (up to scalar multiples) and that $R_{\vc s_2}$ has integer entries.  Thus $G\leq \O_{\LL}^+$.  It is not difficult to show $G=\O_{\LL}^+$ \cite{Bar18}.  The group $\CC$ is generated by reflections across the faces of $\F$ except the face $H_{\vc e_4}$.     

\subsection{The Apollonian/Soddy sphere packing ($\rho=5$)}  
 We select for $G_{\{\vc e_4,\vc e_5\}}$ the fundamental domain $\F_{\{\vc e_4,\vc e_5\}}$ that is the triangle bounded by $\vc v_{12}, \vc v_{23}$ and $\vc u$ shown in \fref{fig3}.  That is, we let $G_{\{\vc e_4,\vc e_5\}}=\langle R_{\vc v_{12}}, R_{\vc v_{23}}, R_{\vc u}\rangle$.  

\myfig[width=\textwidth]{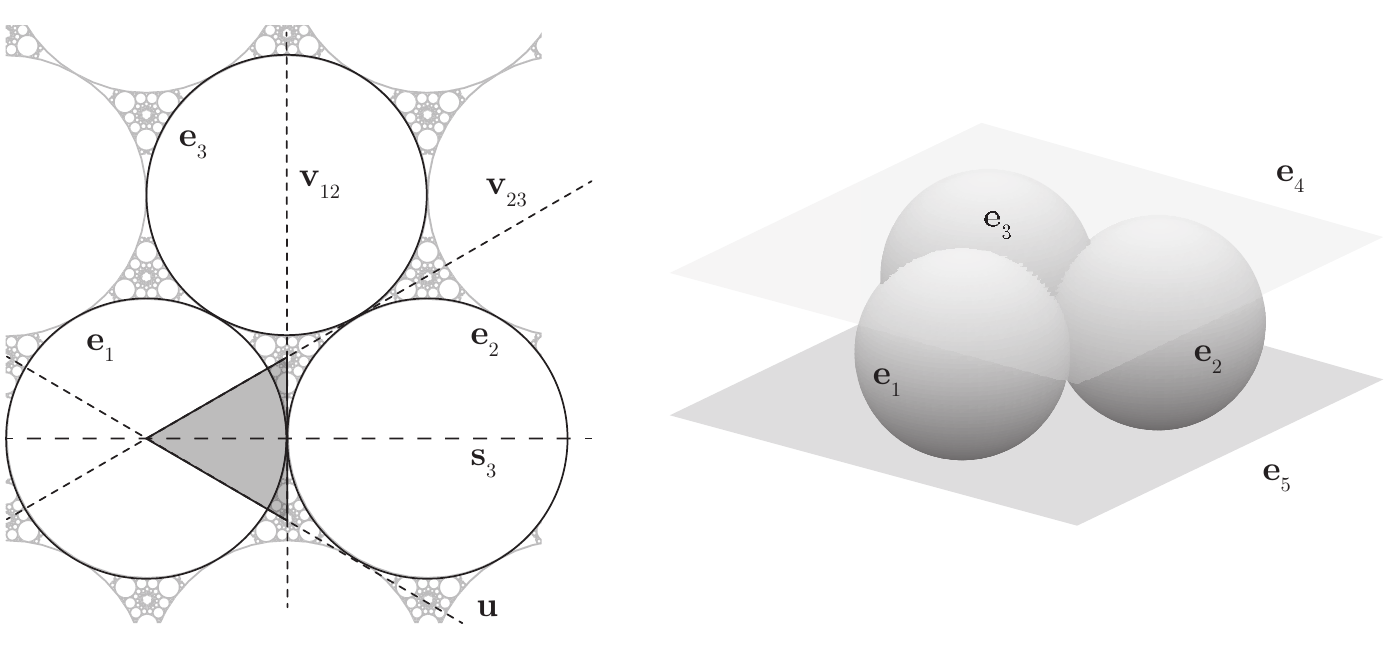}{The cross section on $H_{\vc v_{45}}$ of the strip version of the Apollonian sphere packing with $E=\vc e_4+\vc e_5$ the point at infinity.  Each circle represents a sphere, and the larger spheres are bounded above and below by the planes $H_{\vc e_4}$ and $H_{\vc e_5}$. }    

The \Viete involution $\vc s_3$ shown in \fref{fig3} satisfies $\vc s_3\odot \vc e_i=0$ for $i\neq 3$, and solving we get $\vc s_3=(1,1,-2,1,1)$ up to scalar multiples.  While this is an obvious generalization of the \Viete involution $\vc s_2$ in the $\rho=4$ case, there is a rationale to consider $\vc u$ instead.  We note $\vc u=R_{\vc s_3}(\vc v_{23})=(1,0,-1,1,1)$ and note that $R_{\vc u}\in \O_{\LL}^+$.  In general, we take $\vc u=\vc e_1-\vc e_{\rho-2}+E$, which is $\vc s_2$ when $\rho=4$.  

We extend $G_{\{\vc e_4,\vc e_5\}}$ to $G_E$ by adding the reflections through $H_{\vc e_5}$ and $H_{\vc v_{45}}$.  This stretches the triangle $\F_{\{\vc e_4,\vc e_5\}}$ into a {\it prism}, which forms the base of the chimney.  The chimney is cut off by the dome given by $\vc v_{14}$.  The result is a polyhedron in $\Bbb H^4$ with the single cusp at $E$.  

Algebraically, the details are as follows:  The faces $F_i$ of the triangle are
\begin{alignat*}{2}
(1) &\qquad F_1= H_{\vc v_{12}}: & \qquad x_1-x_2&=0 \\
(2) & \qquad F_2= H_{\vc v_{23}}: & \qquad x_2-x_3&=0 \\
(3)&\qquad F_3= H_{\vc u}: &\qquad x_2+2x_3&=0,
\end{alignat*}
and are on the plane
\[
F_4=H_{\vc v_{45}}: \qquad x_4-x_5=0.
\]
Let $Q_i$ satisfy three of the above four equations, where the ($i$)-th equation is omitted for $i=1, 2, 3$.  Let $Q_i$ also satisfy $Q_i\odot Q_i=0$ and $Q_i\odot D<0$, so $Q_i\in \partial\H$.  (Recall $D=(1,1,1,1,1)$.) %We can use $D=(1,1,1,1,1)$ (and in general, $D=\vc e_1+...+\vc e_\rho$) \cite{Bar18}.  
Then $Q_i$ are the vertices of the triangle $\F_{\{\vc e_4,\vc e_5\}}$, and are, up to positive scalar multiples,
\begin{align*}
Q_1&=(4,0,0,1,1) \\
Q_2&=(8,8,-4,3,3) \\
Q_3&=(4,4,4,-1,-1).
\end{align*}
Let $Q_i'$ be the corresponding vertices of the prism on the face $F_5=H_{\vc e_5}$, so $Q_i$ satisfies equations ($j$) for $j\neq i$, $Q_i'\odot Q_i'=0$, $Q_i'\odot D<0$, and 
\[
F_5=H_{\vc e_5}: \qquad x_1+x_2+x_3+x_4-x_5=0.
\]
Solving, we get 
\begin{align*}
Q_1'&=(1,0,0,0,1) \qquad \hbox{(the point of tangency between $\vc e_1$ and $\vc e_5$)} \\
Q_2'&=(2,2,-1,0,3) \\
Q_3'&=(1,1,1,-1,2). 
\end{align*}
Finally, we select for the last face of $\F$ the plane $F_6=H_{\vc v_{14}}$, where $\vc v_{14}=(1,0,0,-1,0)$.  We verify that $\vc v_{14}\odot E<0$, that $ \vc v_{14}\odot Q_i>0$ for all $i$, and that $\vc v_{14}\odot Q_i'>0$ for all $i$.  Thus the prism lies entirely within the dome $H_{\vc v_{14}}$, so the chimney above the prism and above the dome is a $4$-dimensional polyhedron in $\H$ with only one point $E$ on $\partial \H$.  The group
\[
G=\langle R_{\vc v_{12}}, R_{\vc v_{23}}, R_{\vc u}, R_{\vc v_{45}}, R_{\vc e_5}, R_{\vc v_{14}}\rangle
\]
has fundamental domain $\F$ and because $\F$ has finite volume, $G$ has finite index in $\O_{\LL}^+$.  The group $\CC$ derived from $G$ by removing $R_{\vc e_5}$ from the list of generators is our thin group, and $\CC\F$ is the Apollonian polyhedron whose intersection with $\partial \H$ is the Apollonian/Soddy sphere packing.  

\subsection{Elements in $G$}  In the above, we have identified several types of reflections, namely $R_{\vc e_i}$ for a base element $\vc e_i$; the transpositions $R_{\vc v_{ij}}$ with $\vc v_{ij}=\vc e_i-\vc e_j$;  the \Viete involutions $R_{\vc s_i}$; and the reflection $R_{\vc u}$ for $u=\vc e_1-\vc e_{\rho-2}+E$.  We know $R_{\vc e_i}$ and $R_{\vc v_{ij}}$ are in $\O_{\LL}^+$ for all dimensions; and that the \Viete involutions are not in $\O_{\LL}^+$ for $\rho\geq 6$.  

\begin{lemma}  Let $\vc u=\vc e_i-\vc e_j+E$ for fixed $i,j\leq \rho-2$ and $E=\vc e_{\rho-1}+\vc e_\rho$.  The reflection $R_{\vc u}$ is in $\O_{\LL}^+$ for any $\rho\geq 4$ and fixes both $\vc e_{\rho-1}$ and $\vc e_\rho$.
\end{lemma}

\begin{proof}  Note that $\vc u\odot \vc u=4$, so 
\[
R_{\vc u}(\vc x)=\vc x-2\frac{\vc u\odot \vc x}{4}\vc u.
\]
Thus, it is enough to show that $\vc u\odot \vc x\equiv 0 \pmod 2$, which is not difficult.  However, we will make use of $\vc u\odot \vc x$ in the following, so we calculate it.  We first note that
\[
\vc e_k\odot \vc x=2x_k-\sum_{m=1}^\rho x_m,
\]
so 
\begin{equation}\label{eq3}
\vc u\odot \vc x=2x_i-2x_{j}+2x_{\rho-1}+2x_\rho-2\sum_{m=1}^\rho x_m.  
\end{equation}

In particular, if $\vc x=\vc e_{\rho-1}$ or $\vc e_\rho$ (or $\vc e_i$), then $\vc u\odot \vc x=0$.  Thus both are fixed by $R_{\vc u}$.  
\end{proof}

Another involution of particular value is the map
\begin{equation}\label{ephi}
\phi_{P,E} (\vc x)=\frac{2((P\odot \vc x)E+(E\odot \vc x)P)}{P\odot E}-\vc x,
\end{equation}
for $P$ and $E\in \partial \H$.  On the Euclidean space $\partial \H_E$, it is the $-1$ map centered at $P$.  

\begin{remark}  Suppose $X$ is a K3 surface with elliptic fibration $E$ and section $O$.  For any point $Q$ on $X$, there is a fiber $C\in E$ that contains $Q$.  Let $O_C$ be the point of intersection of the fiber $C$ and the section $O$, and define $\ss(Q)=-Q$ using the group structure of the elliptic curve $C$ with zero $O_C$.  Then $\ss$ is an automorphism of $X$.  If $E$ has maximal rank in $\Pic(X)$, then the pull back $\ss^*$ acting on $\Pic(X)$ is the map $\phi_{P,E}$, where $P=R_{O}(E)$ (the reflection of $E$ in the plane $O\odot \vc x=0$) \cite{Bar17b}.
\end{remark}  
 
\subsection{The cases $\rho=6$, $7$, and $8$} \label{sec34} 
These cases are done the same way as $\rho=5$.  We first find $\F_{\vc e_{\rho-1},\vc e_\rho}$ by looking at the following faces and the equations they give:
\begin{alignat*}{2}
(1)& \qquad F_1=H_{\vc v_{12}}: & \qquad x_1-x_2&=0 \\
(2)&\qquad F_2=H_{\vc v_{23}}: & \qquad x_2-x_3&=0 \\
&\vdots&&\vdots \\
(\rho-3)& \qquad F_{\rho-3}=H_{\vc v_{\rho-3,\rho-2}}: & \qquad x_{\rho-3}-x_{\rho-2}&=0 \\
(\rho-2)&\qquad F_{\rho-2}=H_{\vc u}: &\qquad x_2+x_3+...+x_{\rho-3}+2x_{\rho-2}&=0.
\end{alignat*}
We let $Q_i$ be the point that satisfies all but the ($i$)-th equation above, is on 
\[
F_{\rho-1}=H_{\vc v_{\rho-1,\rho}}: \qquad x_{\rho-1}-x_\rho=0,
\]
and satisfies $Q_i\odot Q_i=0$ and $Q_i\odot D<0$.  These are the vertices of the polytope $\F_{\{\vc e_{\rho-1},\vc e_\rho\}}$.  That we can solve for each $Q_i$, and that the only solution to all the equations is $\vc 0$, shows that the polytope is bounded.  We extend the polytope to the prism $\F_E$ by solving for the $Q_i'$ that satisfies all but the ($i$)-th equation, is on 
\[
F_{\rho}=H_{\vc e_{\rho}}: \qquad x_1+x_2+...+x_{\rho-1}-x_\rho=0,
\]
and satisfies $Q_i'\odot Q_i'=0$, and $Q_i'\odot D<0$.  Over the prism is the chimney.  The dome,
\[
F_{\rho+1}=H_{\vc v_{1,\rho-1}}: \qquad x_1-x_{\rho-1}=0,
\]
has $E$ on one side, and all the $Q_i$'s and $Q_i'$'s on the other side (or in one case, on the plane).  The portion of the chimney outside the dome is $\F$, which intersects $\partial \H$ at the one point $E$, or in the case $\rho=8$, two points $E$ and $Q_4$; its faces give us the reflections that generate $G$; and omitting the reflection $R_{\vc e_\rho}$ gives us $\CC$.  This is slightly slicker than what was done in \cite{Bar18}, as the fundamental domain $\F_E$ is more compact, so we need only one dome to cover it.  

\begin{remark}  For $\rho=8$, the fundamental domain $\F$ has two cusps, one at $E$ and the other at $Q_4$.  Since the face $F_8=H_{\vc e_8}$ does not go through $Q_4$, the stabilizer $\GG_{Q_4}$ has a full rank of six (meaning $\GG_{Q_4}$ includes a subgroup isomorphic to $\Bbb Z^6$).  Thus, in the perspective $\A_{8,Q_4}$, the Apollonian packing includes a six-dimensional Euclidean lattice of spheres.  For lower dimensions, the best we can get is an $N-1=\rho-3$ dimensional lattice (see for example \fref{fig3}).    
\end{remark}

\section{Coxeter graphs (an aside) }\label{sCoxeter}

The groups $\CC=\CC_\rho$ for $\rho\leq 8$ are generated with reflections, so we can derive their Coxeter graphs.  Coxeter graphs are not unique for a packing, since there may be many groups that generate the same packing.  For example, the three groups $\CC$, $\CC'$, and $\CC_\Ap$ that we saw in Section \ref{s2.1} for the Apollonian circle packing ($\rho=4$) have the Coxeter graphs shown in \fref{fig4}.  Maxwell gives 13 different Coxeter graphs that all generate the Apollonian/Soddy sphere packing \cite[Table I]{Max81}.  
The groups $\CC_{\rho}$ that we chose for $\rho=5, ..., 8$ have a nice symmetry to their Coxeter graphs -- see \fref{fig5}.

\myfig{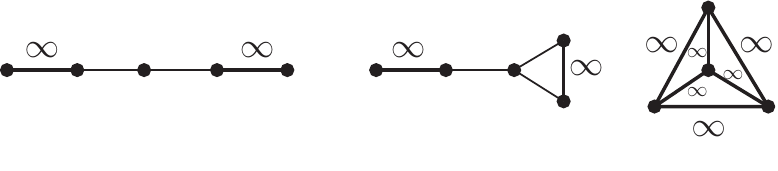}{The Coxeter graphs for the groups $\CC$, $\CC'$, and $\CC_\Ap$ for the Apollonian circle packing ($\rho=4$).}

\myfig{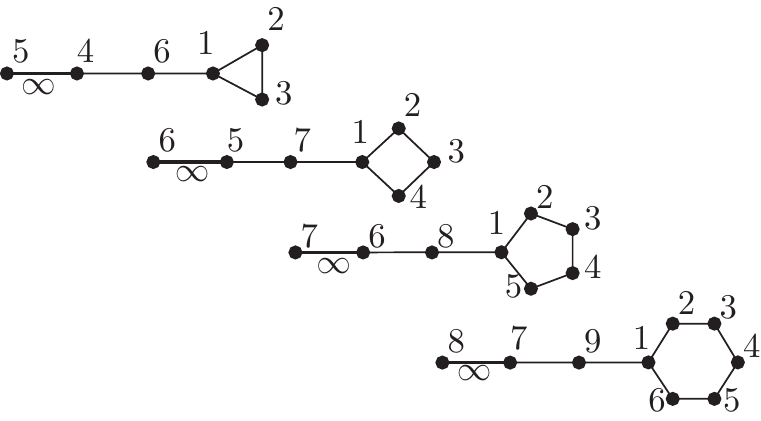}{The Coxeter graphs for $\CC_{\rho}$ with $\rho=5$, $6$, $7$, and $8$.  A vertex labeled $i$ represents the face $F_i$. 
 Note that $F_i\odot F_i=4$ for $i\neq \rho$, and $F_\rho\odot F_\rho=1$.  }  
%This information, together with the Coxeter graph, is enough to recover the lattice.
%{Any subset of $\rho$ faces that includes $F_\rho$ and $F_{\rho+1}$ from a basis of the lattice.}  

%\begin{remark} Showing that $G$ has finite index in $\O_{\LL}^+$ and deriving a Coxeter graph that includes a leaf with a bold stem (subscripted with $\infty$) implies that the image of the face corresponding to the leaf under the action of the group generated by reflections in the other faces generates an Apollonian packing (i.e. a configuration of hyperspheres that satisfies Properties (a) and (b)) \cite{Max81}.  *** Verify this ***
%\end{remark}  

\section{The case $\rho=9$}
We solve for $Q_i$ and $Q_i'$ using the equations/faces in Section \ref{sec34}, and find:
\begin{alignat*}{2}
Q_1&=(4,0,0,0,0,0,0,1,1) \qquad \qquad & Q_1'&=(1,0,0,0,0,0,0,0,1) \\
Q_2&=(24,24,-4,-4,-4,-4,-4,15,15) & Q_2'&=(6,6,-1,-1,-1,-1,-1,2,9) \\
Q_3&=(20,20,20,-8,-8,-8,-8,19,19) & Q_3'&=(5,5,5,-2,-2,-2,-2,3,10) \\
Q_4&=(16,16,16,16,-12,-12,-12,19,19) &\qquad Q_4'&=(4,4,4,4,-3,-3,-3,3,10) \\
Q_5&=(12,12,12,12,12,-16,-16,15,15) & Q_5'&=(3,3,3,3,3,-4,-4,2,9) \\
Q_6&=(8,8,8,8,8,8,-20,7,7) & Q_6'&=(2,2,2,2,2,2,-5,0,7) \\
Q_7&=(4,4,4,4,4,4,4,-5,-5) & Q_7'&=(1,1,1,1,1,1,1,-3,4).
\end{alignat*}
The dome $F_{10}=H_{\vc v_{18}}$ includes all of these points except for $Q_4$ and $Q_5$.  The midpoint of $Q_4Q_5$ (in $\partial \H_E$) is $P_1=(1,1,1,0,-1,-1,1,1)$, which lies on $F_{10}$.  This suggests that $P_1$ is a cusp, and though the faces wall off all but one dimension, there do not seem to be any reflections that close it off.  However, $\phi_{P_1,E}\in \O_{\LL}^+$ (see Equation \ref{ephi}).  While it might be natural to test whether $\phi_{P_1,E}$ is in $\O_{\LL}^+$, it is not so obvious to look at $\phi_{P_1,P_2}$ for $P_2=(1,1,1,1,2,-1,-2,1,2)$, which is also in $\O^+_{\LL}$.  We try to motivate this second choice in the remark at the end of this section.

\begin{theorem}  The reflections through the faces $F_1$, ..., $F_{10}$ and the maps $\phi_{P_1,E}$ and $\phi_{P_1,P_2}$ generate a subgroup $G=G_9$ of finite index in $\O_{\LL_9}^+$.  
\end{theorem}

\begin{proof}  As before, we construct a fundamental domain $\F$ for $G$ and show that its intersection with $\partial \H$ is a finite set of points, and hence has finite volume.  The prism $\P$ with faces $F_1$ through $F_{9}$ is not $\F_E$ since $\phi_{P_1,E}$ is also in $\CC_E$, though we can choose $\F_E$ to be a subset of $\P$.
%, and hence the chimney, so it is enough to find several domes that together cover the prism.  
The map $\phi_{P_1,E}$ is the $-1$ map through $P_1$ on $\partial \H_E$, so any plane through $P_1$ and $E$ can serve as a face, since $\phi_{P_1,E}$ sends all points on one side of the plane to points on the other side, and vice versa.  Let us take $\vc n_1=(1,1,1,1,-6,1,1,1,1)$, which satisfies $\vc n_1\odot Q_4=\vc n_1\odot E=0$ as desired, and let $F_{11}=H_{\vc n_1}$.  %Then $F_{11}$ is a Euclidean plane in $\partial \H_E$ since it goes through $E$.  
We verify that the points $Q_1$ and $Q_1'$ lie on $F_{11}$, the points $Q_2$, $Q_3$, $Q_4$, $Q_2'$, $Q_3'$, and $Q_4'$ lie in $H_{\vc n_1}^+$, and the rest lie in $H_{\vc n_1}^-$.  That is, $F_{11}$ slices $\P$ vertically into two pieces.  We take $\F_E$ to be the prism formed by the intersection of $\P$ with $H_{\vc n_1}^-$ and use this to give us our chimney.  It is now enough to find several domes that cover $\F_E$.        

The dome $F_{10}=H_{\vc v_{18}}$ covers all the vertices of $\P$ except for the two points $Q_4$ and $Q_5$.  The midpoint of $Q_4Q_5$ is $P_1$, which lies on $F_{10}$, so the portion of the  prism $\P$ outside the dome $F_{10}$ has two pieces that touch at $P_1$.  The edges of the prism $\P$ that include $Q_4$ are $Q_4Q_i$ for $i\neq 4$, and $Q_4Q_4'$.  Let $P_{4i}$ be the point closest to $Q_4$ that lies on both $F_{10}$ and $Q_4Q_i$ for $i\neq 4$, and on $Q_4Q_4'$ for $i=4$.  (Note that $P_{45}=P_1$.)  Then the polytope with vertices $Q_4$ and $P_{4i}$ for all $i$ includes one of the two pieces of the portion of the prism $\P$ outside the dome $F_{10}$.  We verify that all these points ($Q_4$ and $P_{4i}$ for all $i$) lie in $H_{\vc n_1}^+$, and therefore this piece is not part of $\F_E$.  

We similarly define $P_{5i}$ (where again $P_{54}=P_1$).  For the map $\phi_{P_1,P_2}$ we can take any plane $H_{\vc n_2}$ that includes $P_1$ and $P_2$.  Let us take $\vc n_2=(3,3,3,3,3,-4,-4,3,3)$ and define $F_{12}=H_{\vc n_2}$ (which is tangent to $H_{\vc n_1}$ at $P_1$ in $\partial \H_E$).  We verify that $E$ is in $H_{\vc n_2}^-$, while the points $Q_5$ and $P_{5i}$ lie in $H_{\vc n_2}^+$.   Thus, the domes $F_{10}$ and $F_{12}$ cover $\F_E$.  Hence, the region $\F$ bounded by the faces $F_1$ through $F_{12}$ (as described above) is a fundamental domain for $G$.  It intersects $\partial \H$ at  $E$ and $P_1$, so has finite volume.  Hence $G$ has finite index in $\O_{\LL}^+$.     
\end{proof}

Our argument from here is like that in \cite{Bar18}, and after establishing Theorem \ref{t6.1} in the next section, completes the argument for $\rho=10$ as well.  We let 
\[
\K'_\rho=\bigcap_{\cc\in \CC_\rho}H^-_{\cc \vc e_\rho}
\]
and show $\K'_\rho=\K_\rho$.  We will need the following result from \cite{Bar18} (which establishes Property (c)):

\begin{lemma}[Lemma 5.1 of \cite{Bar18}]  \label{l5.1old}The planes $H_{\vc e_i}$ are faces of $\K$.  \end{lemma}

The following is an unsurprising density result:

\begin{lemma}  \label{l5.3}  Suppose $D\in H_{\vc n}^-$ for some $\vc n\in \Bbb R^{\rho-1,1}$ with $\vc n\odot \vc n>0$.  Then there exists $\vc m\in \E_1$ such that $H_{\vc m}^+\subset H_{\vc n}^+$.  
\end{lemma}

That is, given an arbitrary hyperball $H_{\vc n}$ in $\partial H_E$, there exists an $\vc m\in \E_1$ so that the hyperball $H_{\vc m}$ is contained in $H_{\vc n}$.   

\begin{proof}  Since $G$ has finite index in $\O_{\LL}^+$, the image of $E$ under the action of $G$ is dense in $\partial \H$.  Thus, there exists a $\cc\in G$ such that $\cc E\in H_{\vc n}^+$.  Without loss of generality, we may also choose $\cc$ so that $\cc E\neq E$ and $\cc E$ is not on $H_{\vc n}$.  Let $E'=\cc E$ and $\vc f=\cc \vc e_\rho$.   Consider $\vc m=\vc m(a)= \vc f+aE'$.  Then 
\begin{align*}
\vc m\odot \vc m&=\cc \vc e_\rho\odot \cc \vc e_\rho+2a\cc\vc e_\rho\odot \cc E+\cc  E\odot \cc E \\
&=\vc e_\rho \odot \vc e_\rho+2a\vc e_\rho\odot E+E\odot E \\
&=1.
\end{align*}
Note that $\vc m(a)\odot E'=\cc \vc e_\rho \odot \cc E=\vc e_\rho\odot E=0$, so $E'$ lies on $H_{\vc m}$.  The curvature of $H_{\vc m}$ is
\[
-\vc m(a)\odot E= \vc f \odot E +aE'\odot E.
\]
Since $E'\neq E$, we know $E'\odot E\neq 0$, so with an appropriately large choice of $|a|$ and appropriate sign, we can make the hyperball given by $\vc m=\vc m(a)$ have a sufficiently small radius that $H_{\vc m}^+\subset H_{\vc n}^+$.  Since $D\in H_{\vc n}^-\subset H_{\vc m}^-$, we get $\vc m\in \E_1$.  
\end{proof}

\begin{theorem}  For $\rho=9$ and $10$ (and also $4$ through $8$), $\K_\rho'=\K_\rho$, and its intersection $\A_{\rho,E}$ with $\partial \H_E$ satisfies Properties (a), (b), and (d).
\end{theorem}

\begin{proof}  We begin by showing $\K'=\K$.  Since $\CC\vc e_\rho\subset \E_1$, we know $\K'\supset \K$.  

Suppose there exists a face $H_{\vc m}$ of $\K$ that is not a face of $\K'$, with $\vc m\in \E_1$.   We use $\CC$ to do a descent on $\vc m$:  Let $O\in \partial \H_E$ be the center of the hypersphere $H_{\vc m,E}\subset \partial \H_E$.  There exists an isometry in $\CC_E$ that moves $O$ into $\F_E$, so $O$ is inside one (or more) of the domes that bound $\F$ ($F_{10}$ and $F_{12}$ for $\rho=9$; $F_{11}$, $F_{12}$, and $F_{14}$ for $\rho=10$).  Note that this map does not change the curvature, as $\CC_E$ is a group of Euclidean isometries of $\H_E$.  Since $O$ is in one of the domes, applying the map associated to that dome strictly decreases the curvature.  We find the new $O$ (if  it exists) and continue.  Since the curvatures are integers, descent must stop, which is when the curvature reaches $0$ (so $O$ no longer exists).  Let $\vc m'$ be the final image of $\vc m$.  Then $\vc m'\odot E=0$ (curvature is $0$), so $H_{\vc m'}$ is either parallel to $H_{\vc e_\rho}$, or intersects it.   

If $H_{\vc m'}$ is parallel to $H_{\vc e_\rho}$, then it is also parallel to $H_{\vc e_{\rho-1}}$ and hence $H_{\vc m'}^-$ is a subset of either $H_{\vc e_\rho}^-$ or $H_{\vc e_{\rho-1}}^-$.  That means one of them is not a face of $\K$, which contradicts Lemma \ref{l5.1old}.      

If $H_{\vc m'}$ intersects $H_{\vc e_\rho}$, then $\vc m'\odot \vc e_\rho=0$, since this product is an integer in the interval $(-1,1)$ (see Equation (\ref{eq1})).  Thus
\[
0=\vc m'\odot \vc e_\rho=m_\rho'-\sum_{i=1}^{\rho-1}m_i',
\]
and
\begin{align*}
1=\vc m'\odot \vc m'&=\sum_{i=1}^\rho (m_i')^2-2\sum_{i < j} m_i'm_j' \\
&\equiv \left(\sum_{i=1}^\rho m_i\right)^2 \pmod 2 \\
&\equiv (\vc m'\odot \vc e_\rho)^2 \equiv 0 \pmod 2,
\end{align*}
a contradiction.  Thus, no such $\vc m$ exists and $\K'=\K$.  

This is also the key to showing Property (a):  Suppose two faces, $\vc m$ and $\vc m'\in \CC\vc e_\rho$ intersect.  Then $\vc m\odot \vc m'=0$, and there exists a $\cc\in \CC$ so that $\cc \vc m'=\vc e_\rho$.  Thus $\cc\vc m\odot \vc e_\rho=0$, leading us to a contradiction using the logic above.  

We next show $\A_{\rho,E}$ satisfies Property (b):  Suppose not.  Then there exists a gap where we can place a hypersphere.   That is, there exists $\vc n\in \Bbb R^{\rho-1,1}$ with $\vc n\odot \vc n>0$ so that $H_{\vc n}^+\subset \K$.  If $D\in H_{\vc n}^+$, then the hypersphere is very large, and by shrinking it (choosing a different $\vc n$), we can make sure $D\in H_{\vc n}^-$.  (The point $D$ lies in $\H$, so is a fixed Euclidean distance above $\partial \H_E$.  If the hypersphere in $\partial \H_E$ is sufficiently small, then the hemisphere above it cannot reach $D$.)  By Lemma \ref{l5.3}, there exists $\vc m\in \E_1$ so that $H_{\vc m}^+\subset H_{\vc n}^+\subset \K$.  We do the descent argument again to get $\vc m'$ with $\vc m'\odot E=0$.  Since $\CC$ is a group of symmetries of $\K$, we now have a gap in $\K$ that includes the half space $H_{\vc m'}^+$ in the Euclidean space $\partial \H$.   This cannot happen, as $\K$ lies between the parallel planes $H_{\vc e_\rho}$ and $H_{\vc e_{\rho-1}}$.  

Finally, Property (d) follows from Lemma \ref{l5.1old} and the transitivity of $\CC$ on the faces of $\K$.     
\end{proof}

\begin{remark}  Here is a perhaps dubious motivation for looking at $\phi_{P_1,P_2}$.  Note that if a hypersphere intersects a face $F_i$ of $\F_E$, then the reflection of it in $F_i$ is itself (if we want Property (b) to hold).  Since we want to cover $Q_5$, a strategy is to look for elements $\vc n\in \E_1$ on the intersection of all but one of the faces that go through $Q_5$ and are close to $Q_5$.  We stumble on $\vc f=(2,2,2,2,2,-2,-3,2,2)$.  We also note that $P_1$ is a cusp, so there is some motivation to make it the point at infinity.  With $P_1$ the point at infinity, we note that both $\vc e_9$ and $\vc f$ have curvature $1$, so we look at the $-1$ map through the midpoint $P_2$ of the two.  To find $P_2$, we find the center $R_{\vc e_9}(P_1)$ of $H_{\vc e_9,P_1}$ in $\partial \H_{P_1}$, and the center $R_{\vc f}(P_1)$ of $H_{\vc f,P_1}$.  The point $P_2$ is the sum of these two centers together with the appropriate multiple of $P_1$ that makes $P_2\odot P_2=0$.      
\end{remark}  

\begin{remark}  Descent arguments usually include a height function.  Our method of descent has two components: Moving $H_{\vc m}$ closer to $\F_E$ using the Euclidean metric on $\partial \H_E$, while keeping the curvature constant; and decreasing the curvature.  Thus, our height can be thought of as having two components:  Euclidean distance of $H_{\vc m}$ from (say) $Q_1=R_{\vc e_1}(E)$, which is essentially the quantity $-Q_1\odot \vc m$; and curvature, which is essentially $-E\odot \vc m$.  We can combine these two to give us a more traditional height:  $h(\vc m)=-(Q_1+3E)\odot \vc m$.  The coefficient $3$ is necessary so that the point $Q_1+3E$ is in $\F$ and in $\H$.  Otherwise, the decrease in curvature after an inversion might be less than the increase in distance.  This height is essentially just the logarithm of the hyperbolic distance from the point $Q_1+3E$ to the plane $H_{\vc m}$ in $\H$.  
\end{remark}

\section{The case $\rho=10$}

As in the previous case, we use the equations in Section \ref{sec34} to find the vertices of a prism $\P$, the points $Q_1$, ..., $Q_8$ and $Q_1'$, ..., $Q_8'$.  The dome $F_{11}=R_{\vc v_{19}}$ covers all but $Q_4$, $Q_5$, and $Q_6$.  The point $Q_5'$ is on $F_{11}$.  In our search for additional isometries, we find the reflection $R_{\vc n}$ for $\vc n=(1,1,1,1,1,-1,-1,-1,1,1)$ is in $\O_{\LL}^+$.  We let $F_{12}=H_{\vc n}$.  The hypersphere $H_{\vc n,E}$ in $\H_E$ is centered at $Q_5$ and goes through $Q_5'$.  The midpoint $P_1=(1,1,1,1,0,0,-1,-1,1,1)$ of $Q_4Q_6$ in $\H_E$ is on $F_{11}$ and $F_{12}$, so looks like a cusp.  The map $\phi_{P_1,E}$ is in $\O_{\LL}^+$.  We find $\vc f=(1,1,1,1,1,1,-1,-2,1,1)\in \E_1$ in roughly the same place as the $\rho=9$ case, and in the same way find $P_2=(1,1,1,1,2,2,-1,-3,1,3)$.  The map $\phi_{P_1,P_2}$ is also in $\O_{\LL}^+$.  

\begin{theorem}\label{t6.1}   The reflections through the faces $F_1$, ..., $F_{12}$ and the maps $\phi_{P_1,E}$ and $\phi_{P_1,P_2}$ generate a subgroup $G=G_{10}$ of finite index in $\O_{\LL_{10}}^+$.  
\end{theorem}

\begin{proof}  As in the $\rho=9$ case, we let $\P$ be the prism with vertices $Q_1$, ..., $Q_8$ and $Q_1'$, ..., $Q_8'$.  Because the map $\phi_{P_1,E}$ is in $\CC_E$, our choice for $\F_E$ will be a proper subset of $\P$.  We choose the face $F_{13}=H_{\vc n_1}$ where $\vc n_1=(1,1,1,1,-3,-3,1,1,1,1)$ to represent $\phi_{P_1,E}$.  Since $P_1$ and $E$ are both on $F_{13}$, it sends one side of the face to the other.  We let $\F_E$ be the intersection of $\P$ with $H_{\vc n_1}^-$.  This lops off the vertices $\{Q_2,Q_3,Q_4,Q_2',Q_3',Q_4'\}$.  The vertices $\{Q_1,Q_1',Q_5,Q_5'\}$ are on $F_{13}$.  Since $F_{13}$ is perpendicular to $H_{\vc e_{10}}$ and $H_{\vc v_{9,10}}$, the domain $\F_E$ is again a prism.  Let $P_{ij}$ be the intersection of the line $Q_iQ_j$ with $F_{13}$ for $i=6$, $7$, and $8$, and $j=2$, $3$, and $4$.  Define $P_{ij}'$ similarly for the line $Q_i'Q_j'$ and the same indices.  Then the vertices of $\F_E$ are the vertices $Q_1$, $Q_5$, $Q_6$, $Q_7$, $Q_8$, the nine vertices $P_{ij}$ just defined, and the `prime' versions of those fourteen points.  Most of $\F_E$ is covered by the dome $F_{11}$, as that dome includes all the vertices of $\F_E$ except for $Q_5$ and $Q_6$.

The dome $F_{12}$ covers the vertex $Q_5$, but not $Q_6$.  For $\phi_{P_1,P_2}$, we choose $\vc n_2=(3,3,3,3,3,3,-5,-5,3,3)$ and the face $F_{14}=H_{\vc n_2}$, which covers $Q_6$.   

A convex polyhedron is covered by a single sphere if all of its vertices are in the sphere.  We used this in the $\rho\leq 8$ cases.  It is covered by two spheres if all the vertices are covered, and all the edges are covered.  If two vertices are in the same sphere, then that edge is covered by that sphere, so one need check only the edges with vertices in different spheres.  A convex polyhedron is covered by three spheres if all its vertices and edges are covered, and also all its $2$-D edges are covered.  If three vertices that lie in two spheres define a 2-D edge, then it is covered if its edges are covered.  Thus we need only check 2-D edges with vertices in all three spheres.  This is our situation.

The edge $AB$ in $\partial \H_E$ is in the plane spanned by $\{A,B,E\}$ in $\Bbb R^{\rho-1,1}$.  To see if it is covered by the two domes that contain $A$ and $B$, we solve $P\odot F=0$ for $P\in \myspan\{A,B,E\}$ and the two domes $F$.  This gives us a one-dimensional subspace of $\Bbb R^{\rho-1,1}$ spanned by (say) $P$.  If $P\odot P> 0$, then the domes intersect above the segment $AB$ in the Poincar\'e model, so the spheres in $\partial H_E$ cover the edge.  If $P\odot P=0$, then the domes intersect at $P$ on the edge, so again cover it.  They do not cover the edge if $P\odot P<0$.  A similar argument works for a 2-D edge determined by vertices $A$, $B$ and $C$:  We let $P\in \myspan\{A,B,C,E\}$ and solve the three equations $F\odot P=0$ for the domes $F$ that cover each of the three points.  This gives a one-dimensional solution spanned by $P$, and again, if $P\odot P\geq0$, then the 2-D edge is covered by the three spheres, and does not otherwise.  

Let $S=\{Q_1,Q_7, Q_8,P_{62},P_{63}, P_{64}, P_{72}, P_{73}, P_{74}, P_{82}, P_{83}, P_{84}\}$ be the set of vertices that are on the top of the prism $\F_E$, and that are covered by $F_{11}$.  We check that the edges $Q_5A$ for $A\in S \cup \{Q_5'\}$  are all covered by $F_{11}$ and $F_{12}$.  We do not need to check, for example, the edge $Q_5Q_1'$, since it follows from our check of $Q_5Q_5'$ and $Q_5Q_1$, and noting that $\F_E$ is a prism.  We similarly check that the edges $Q_6A$ for $A\in S\cup \{Q_6'\}$ are covered by $F_{11}$ and $F_{14}$; and that $Q_5Q_6$ is covered by $F_{12}$ and $F_{14}$.  Finally we check that all the 2-D edges spanned by $\{Q_5,Q_6,A\}$ for $A\in S\cup \{Q_5',Q_6'\}$ are covered by $F_{11}$, $F_{12}$ and $F_{14}$.  We find that they all are covered by $F_{11}$, $F_{12}$, and $F_{14}$, so $\F$ has finite volume and thus $G$ has finite index in $\O^+$. 
\end{proof}


\begin{bibdiv}

\begin{biblist}

\ignore{\bib{A-V93}{article}{
   author={Alekseevskij, D. V.},
   author={Vinberg, \`E. B.},
   author={Solodovnikov, A. S.},
   title={Geometry of spaces of constant curvature},
   conference={
      title={Geometry, II},
   },
   book={
      series={Encyclopaedia Math. Sci.},
      volume={29},
      publisher={Springer, Berlin},
   },
   date={1993},
   pages={1--138},
   review={\MR{1254932}},
   doi={10.1007/978-3-662-02901-5\_1},
}

\bib{Apa00}{book}{
   author={Apanasov, Boris N.},
   title={Conformal geometry of discrete groups and manifolds},
   series={De Gruyter Expositions in Mathematics},
   volume={32},
   publisher={Walter de Gruyter \& Co., Berlin},
   date={2000},
   pages={xiv+523},
   isbn={3-11-014404-2},
   review={\MR{1800993}},
   doi={10.1515/9783110808056},
}
	
\bib{Bar01}{book}{
   author={Baragar, Arthur},
   title={A Survey of Classical and Modern Geometries},
   publisher={Prentice Hall},
   date={2001},
   address={Upper Saddle River, NJ},
   pages={xiv+370},
   isbn={0-13-014318-9},
}

\bib{Bar11}{article}{
   author={Baragar, Arthur},
   title={The ample cone for a $K3$ surface},
   journal={Canad. J. Math.},
   volume={63},
   date={2011},
   number={3},
   pages={481--499},
   issn={0008-414X},
   review={\MR{2828530 (2012f:14071)}},
   doi={10.4153/CJM-2011-006-7},
}



\bib{Bar17}{article}{
  author={Baragar, Arthur},
  title={The Apollonian circle packing and ample cones for K3 surfaces},
  eprint={arXiv:1708.06061},
  status={to appear},
  year={2017},
}}

\bib{Bar17b}{article}{
  author={Baragar, Arthur},
  title={The Neron-Tate pairing and elliptic K3 surfaces},
  eprint={arXiv:1708.05998},
  status={to appear},
  year={2017},
}

\bib{Bar18}{article}{
   author={Baragar, Arthur},
   title={Higher dimensional Apollonian packings, revisited},
   journal={Geom. Dedicata},
   volume={195},
   date={2018},
   pages={137--161},
   issn={0046-5755},
   review={\MR{3820499}},
   doi={10.1007/s10711-017-0280-7},
}

\ignore{
\bib{Boy73}{article}{
   author={Boyd, David W.},
   title={The osculatory packing of a three dimensional sphere},
   journal={Canad. J. Math.},
   volume={25},
   date={1973},
   pages={303--322},
   issn={0008-414X},
   review={\MR{0320897}},
   doi={10.4153/CJM-1973-030-5},
}}

\bib{Boy74}{article}{
   author={Boyd, David W.},
   title={A new class of infinite sphere packings},
   journal={Pacific J. Math.},
   volume={50},
   date={1974},
   pages={383--398},
   issn={0030-8730},
   review={\MR{0350626}},
}

\ignore{
\bib{Boy82}{article}{
   author={Boyd, David W.},
   title={The sequence of radii of the Apollonian packing},
   journal={Math. Comp.},
   volume={39},
   date={1982},
   number={159},
   pages={249--254},
   issn={0025-5718},
   review={\MR{658230}},
   doi={10.2307/2007636},
}

\bib{Cli68}{article}{
  author={Clifford, W.H.},
  title={On the powers of spheres (1868)},
  booktitle={Mathematical papers},
  year={1882},
  publisher={Macmillan},
  address={London},
}

\bib{Dar72}{article}{
   author={Darboux, Gaston},
   title={Sur les relations entre les groupes de points, de cercles et de
   sph\`eres dans le plan et dans l'espace},
   language={French},
   journal={Ann. Sci. \'Ecole Norm. Sup. (2)},
   volume={1},
   date={1872},
   pages={323--392},
   issn={0012-9593},
   review={\MR{1508589}},
}

\bib{Dol16}{article}{
  author={Dolgachev, Igor},
  title={Orbital counting of curves on algebraic surfaces and sphere packings},
  booktitle={K3 surfaces and their moduli},
  year={2016},
  publisher={Springer International Publishing},
  address={Cham},
  pages={17--53},
  isbn={978-3-319-29959-4},
  doi={10.1007/978-3-319-29959-4\_2},
  url={http://dx.doi.org/10.1007/978-3-319-29959-4_2},
}}

\bib{GLM06}{article}{
   author={Graham, Ronald L.},
   author={Lagarias, Jeffrey C.},
   author={Mallows, Colin L.},
   author={Wilks, Allan R.},
   author={Yan, Catherine H.},
   title={Apollonian circle packings: geometry and group theory. II.
   Super-Apollonian group and integral packings},
   journal={Discrete Comput. Geom.},
   volume={35},
   date={2006},
   number={1},
   pages={1--36},
   issn={0179-5376},
   review={\MR{2183489}},
   doi={10.1007/s00454-005-1195-x},
}

\ignore{
\bib{G-M10}{article}{
   author={Guettler, Gerhard},
   author={Mallows, Colin},
   title={A generalization of Apollonian packing of circles},
   journal={J. Comb.},
   volume={1},
   date={2010},
   number={1, [ISSN 1097-959X on cover]},
   pages={1--27},
   issn={2156-3527},
   review={\MR{2675919}},
   doi={10.4310/JOC.2010.v1.n1.a1},
}
	
\bib{Kov94}{article}{
   author={Kov{\'a}cs, S{\'a}ndor J.},
   title={The cone of curves of a $K3$ surface},
   journal={Math. Ann.},
   volume={300},
   date={1994},
   number={4},
   pages={681--691},
   issn={0025-5831},
   review={\MR{1314742 (96a:14044)}},
   doi={10.1007/BF01450509},
}}

\bib{LMW02}{article}{
   author={Lagarias, Jeffrey C.},
   author={Mallows, Colin L.},
   author={Wilks, Allan R.},
   title={Beyond the Descartes circle theorem},
   journal={Amer. Math. Monthly},
   volume={109},
   date={2002},
   number={4},
   pages={338--361},
   issn={0002-9890},
   review={\MR{1903421}},
   doi={10.2307/2695498},
}
\bib{Max81}{article}{
   author={Maxwell, George},
   title={Space groups of Coxeter type},
   booktitle={Proceedings of the Conference on Kristallographische Gruppen
   (Univ. Bielefeld, Bielefeld, 1979), Part II},
   journal={Match},
   number={10},
   date={1981},
   pages={65--76},
   issn={0340-6253},
   review={\MR{620801}},
}

\ignore{
\bib{McM98}{article}{
   author={McMullen, Curtis T.},
   title={Hausdorff dimension and conformal dynamics. III. Computation of
   dimension},
   journal={Amer. J. Math.},
   volume={120},
   date={1998},
   number={4},
   pages={691--721},
   issn={0002-9327},
   review={\MR{1637951}},
}}

\bib{Mor84}{article}{
   author={Morrison, D. R.},
   title={On $K3$ surfaces with large Picard number},
   journal={Invent. Math.},
   volume={75},
   date={1984},
   number={1},
   pages={105--121},
   issn={0020-9910},
   review={\MR{728142}},
   doi={10.1007/BF01403093},
}

\bib{Rat06}{book}{
   author={Ratcliffe, John G.},
   title={Foundations of hyperbolic manifolds},
   series={Graduate Texts in Mathematics},
   volume={149},
   edition={2},
   publisher={Springer, New York},
   date={2006},
   pages={xii+779},
   isbn={978-0387-33197-3},
   isbn={0-387-33197-2},
   review={\MR{2249478}},
}

\ignore{
\bib{Sod37}{article}{
  author={Soddy, Frederick},
  title={The bowl of integers and the hexlet},
  journal={Nature},
  volume={139},
  year={1937-01-09},
  pages={77--79},
  doi={10.1038/139077a0}
}

\bib{Sul84}{article}{
   author={Sullivan, Dennis},
   title={Entropy, Hausdorff measures old and new, and limit sets of
   geometrically finite Kleinian groups},
   journal={Acta Math.},
   volume={153},
   date={1984},
   number={3-4},
   pages={259--277},
   issn={0001-5962},
   review={\MR{766265}},
   doi={10.1007/BF02392379},
}

\bib{V-S93}{article}{
   author={Vinberg, \`E. B.},
   author={Shvartsman, O. V.},
   title={Discrete groups of motions of spaces of constant curvature},
   conference={
      title={Geometry, II},
   },
   book={
      series={Encyclopaedia Math. Sci.},
      volume={29},
      publisher={Springer, Berlin},
   },
   date={1993},
   pages={139--248},
   review={\MR{1254933}},
   doi={10.1007/978-3-662-02901-5\_2},
}	}	

\end{biblist}
\end{bibdiv}
\end{document}